\documentclass[oneside,english]{amsart}
\usepackage[T1]{fontenc}
\usepackage[latin9]{inputenc}
\usepackage{geometry}
\usepackage{amsthm}
\usepackage{amstext}
\usepackage{amssymb}
\usepackage{esint}

\makeatletter
\numberwithin{equation}{section}
\numberwithin{figure}{section}

\@ifundefined{date}{}{\date{}}
\theoremstyle{plain}
\newtheorem{thm}{\protect\theoremname}
  \theoremstyle{definition}
    \newtheorem{defn}{\protect\definitionname}
  \theoremstyle{plain}

  \theoremstyle{remark}
  \newtheorem*{rem*}{\protect\remarkname}
  \newtheorem*{rem}{\protect\remarkname}
  \theoremstyle{definition}
  \newtheorem{example}{\protect\examplename}

\numberwithin{equation}{section} %% Comment out for sequentially-numbered
\numberwithin{figure}{section} %% Comment out for sequentially-numbered
\usepackage{babel}
\providecommand{\definitionname}{Definition}
  \providecommand{\examplename}{Example}
  \providecommand{\remarkname}{Remark}
\providecommand{\theoremname}{Theorem}

\usepackage{babel}
\providecommand{\definitionname}{Definition}
  \providecommand{\examplename}{Example}
  \providecommand{\propositionname}{Proposition}
  \providecommand{\remarkname}{Remark}
\providecommand{\theoremname}{Theorem}

\numberwithin{equation}{section} %% Comment out for sequentially-numbered
\numberwithin{figure}{section} %% Comment out for sequentially-numbered
\usepackage{babel}
\providecommand{\definitionname}{Definition}
  \providecommand{\examplename}{Example}
  \providecommand{\remarkname}{Remark}
\providecommand{\theoremname}{Theorem}

\usepackage{babel}
\providecommand{\definitionname}{Definition}
  \providecommand{\examplename}{Example}
  \providecommand{\propositionname}{Proposition}
  \providecommand{\remarkname}{Remark}
\providecommand{\theoremname}{Theorem}

\usepackage{babel}

\makeatother

\usepackage{babel}
\begin{document}

\title{Alternative Lagrangians obtained by scalar deformations}

\author[Constantinescu]{Oana A. Constantinescu}
\address{Faculty of Mathematics \\
 Alexandru Ioan Cuza University \\
 Bd. CAROL I, nr. 11, 700506, Ia\c{s}i, Romania}

\email{oanacon@uaic.ro}

\author[Taha]{Ebtsam H. Taha}
\address{Department of Mathematics, Faculty of Science, Cairo University,
12613 Giza, Egypt}

\email{ebtsam@sci.cu.edu.eg, ebtsam.h.taha@hotmail.com}
\begin{abstract}
We study mechanical systems that can be recast into the form of a
system of genuine Euler-Lagrange equations. The equations of motions
of such systems are initially equivalent to the system of Lagrange
equations of some Lagrangian $L$, including a covariant force field.
We find necessary and sufficient conditions for the existence of a
differentiable function $\Phi:\mathbb{R}\rightarrow\mathbb{R}$ such
that the initial system is equivalent to the system of Euler-Lagrange
equations of the deformed Lagrangian $\Phi(L)$. We construct various
examples of such deformations. 
\end{abstract}
\maketitle

\section{Introduction}

When a mechanical system is subject only to conservative forces derived
from a potential, its equations of motions take the form of a system
of second order ordinary differentiable equations (SODE) of the type

\begin{equation}
\frac{d^{2}x^{i}}{dt^{2}}+2G^{i}\left(x,\dot{x}\right)=0,\ i\in\overline{1,n}\label{eq:sode}
\end{equation}

and this system is equivalent to the system of Euler-Lagrange equations
of some Lagrangian $L$, 
\[
\frac{d}{dt}\left(\frac{\partial L}{\partial\dot{x}^{i}}\right)-\frac{\partial L}{\partial x^{i}}=0.
\]
The Lagrangian energy remains constant along the solutions of the
system.

But many mechanical systems do not have a pure Lagrangian description.
Their equations of motions can be written in the form 
\[
\]
\begin{equation}
\frac{d}{dt}\left(\frac{\partial L}{\partial\dot{x}^{i}}\right)-\frac{\partial L}{\partial x^{i}}=\sigma_{i}(x,\dot{x}),\label{eq:eqnoncons}
\end{equation}
where $\sigma_{i}(x,\dot{x})$ is a covariant tensor field that might
represent some external forces. The most studied cases of non-conservative
mechanical systems are dissipative and gyroscopic systems \cite{key-2,key-6,key-9,key-11}.
For example, for classical dissipative systems, we have $\sigma_{i}(x,\dot{x})=\frac{\partial D}{\partial\dot{x}^{i}}$,
with $D$ a quadratic function in the velocities with a negative definite
coefficients matrix. In this case the energy of $L$ decays in time.
In \cite{key-2-1} the authors proved that any SODE on a two dimensional
manifolds is of dissipative type, in a more general sense: $\sigma=\sigma_{i}(x,\dot{x})dx^{i}$
is a $d_{J}$ closed semi-basic $1$-form, where $J$ is the vertical
endomorphism.

In some particular cases there exist several different Lagrangians,
called alternative Lagrangians, such that the system (\ref{eq:sode})
is equivalent to the system of Euler-Lagrange equations of each of
those Lagrangians. Alternative Lagrangian can be used, for example,
to construct constant of motions \cite{key-4}.

One can work with alternative Lagrangians in order to simplify computations.
Given a geometric structure on a differentiable manifold, it is possible
to consider more general structures on the same manifold, related
to the first one, and study the relations between the geometric objects
induced by these different structures. For example, in \cite{key-1},
for a Finsler manifold $(M,F)$ and a differential deformation $\alpha:\mathbb{R}^{+}\rightarrow\mathbb{R}$,
a Lagrangian manifold $\left(M,L=\alpha(F^{2})\right)$ was constructed.
In general the computations in the associated Lagrangian manifold
are simpler and one can obtain interesting information about the initial
Finsler manifold. The examples presented on this article involve Antonelli's
ecological metric and Synge metric \cite{key-0-1,key-0}, with various
applications in biology and physics.

During the last few years non-standard Lagrangians are considered
good candidates to explain non-conservative dynamical systems. Non-standard
Lagrangians are Lagrangians that can not be expressed as differences
between kinetic energy terms and potential energy terms. They can
be of exponential type, of power-law type, or even radical type \cite{C-N}.
They were used, for instance, to study second order Ricati and Abel
equations \cite{CRS} or non-inertial dynamics \cite{key-7-1}.

In a recent paper, \cite{key-3}, starting with a given Lagrangian
$L$, the authors determined all deformations $\Psi:\mathbb{R}\rightarrow\mathbb{R},$
such that $\Psi(L)$ and $L$ give the same dynamical vector field.
This article was the starting point for the study developed here.
This paper was also inspired by \cite{key-0-1,key-1,key-2,C-N,key-7-1,MZE,SA},
in which the authors searched for alternative or non-standard Lagrangians.\\

In this article we are interested in the following situation. We start
with a SODE in normal form (\ref{eq:sode}), which can be written
in an equivalent way as in (\ref{eq:eqnoncons}), for some Lagrangian
$L$ and some semi-basic 1-form $\sigma=\sigma_{i}(x,\dot{x})dx^{i}$.
We emphasize that $\sigma$ is not apriori given. More exactly, given
a SODE (\ref{eq:sode}), we can choose an arbitrary Lagrangian $L$
and find a suitable 1-form $\sigma$ such that (\ref{eq:sode}) and
(\ref{eq:eqnoncons}) are equivalent. Inspired by \cite{key-3}, we
determine necessary and sufficient conditions for the existence of
differentiable deformations $\Phi:\mathbb{R}\rightarrow\mathbb{R}$,
such that the system (\ref{eq:sode}) is equivalent to the system
of Euler-Lagrange equations of the new Lagrangian $\Phi(L)$: 
\begin{equation}
\frac{d}{dt}\left(\frac{\partial\Phi(L)}{\partial\dot{x}^{i}}\right)-\frac{\partial\Phi(L)}{\partial x^{i}}=0.\label{eq:EL|phi(L)}
\end{equation}

The study of the solutions of such a SODE, from geometric point of
view, consists in identifying the system (\ref{eq:sode}) with a {\small{second
order vector field or a semi-spray. This means that $S$ is a vector
field on $TM$, ($M$ an $n$-dimensional differentiable manifold),
with }}$JS=\mathbf{C}$, where $J$ is the vertical endomorphism and
$\mathbf{C}$ the Liouville vector field \cite{key-7}. The solutions
of the above system are called the geodesics of the semispray.

{\small{Locally, $S$ can be represented in the natural basis of $TM$
as follows}} 
\[
S=y^{i}\frac{\partial}{\partial x^{i}}-2G^{i}(x,\dot{x})\frac{\partial}{\partial\dot{x}^{i}}.
\]

We use an intrinsic formulation with no explicit choice of coordinates.
The main result of this article is the following.\\

Theorem 1: Consider a SODE of type (\ref{eq:sode}) and $S\in\mathfrak{X}(TM)$
the associated semi-spray. Suppose that the SODE can be written in
an equivalent way as $\delta_{S}L=\sigma$, for some Lagrangian $L$
and some semi-basic 1-form $\sigma=\sigma_{i}(x,\dot{x})dx^{i}$,
with $S(L)\neq0$ and $\mathbf{C}(L)\neq0$. Then there exists a non
constant, differentiable (of class $C^{2}$) function $\Phi:\mathbb{R}\rightarrow\mathbb{R}$
such that $S$ is a Lagrangian vector field with the corresponding
Lagrangian $\Phi(L)$ if and only if the following three conditions
are satisfied: 
\[
\sigma=\frac{S(E_{L})}{\mathbf{C}(L)}d_{J}L,
\]

\[
\frac{\Phi''(L)}{\Phi'(L)}=-\frac{S(E_{L})}{S(L)\mathbb{\mathbf{C}}(L)},
\]

\[
\left(\Phi''\frac{\partial L}{\partial y^{j}}\frac{\partial L}{\partial y^{i}}+\Phi'\frac{\partial^{2}L}{\partial y^{i}\partial y^{j}}\right)_{i,j}\neq O_{n},
\]
where $O_{n}$ is the null matrix and $E_{L}$ is the energy of $L$.\\

This theorem provides a class of SODEs, of the form 
\[
{\mathcal{L}_{S}}d_{J}L-dL=\frac{S(E_{L})}{\mathbb{\mathbf{C}}(L)}d_{J}L,
\]
which admit a pure Lagrangian description in terms of a deformed Lagrangian
$\Phi(L):\,\,{\mathcal{L}_{S}}d_{J}\Phi(L)-d\Phi(L)=0$.

We give various examples of explicit deformations. Even if most of
the examples are constructed, they are subordinated to classes of
scalar deformations already used in applications in theoretical Physics
or Biology. The examples of the article start with regular Lagrangians
$L$ and include both regular and singular Lagrangians $\Phi(L)$.

We emphasize that Theorem 1 gives a very simple way to check in practice
if the equations of motions of a non-conservative mechanical system
can be recast into the form of genuine Euler-Lagrange equations, with
the new Lagrangian obtained using a scalar deformation of the initial
one. As particular cases, we discussed dissipative and homogeneous
SODEs.

\section{Preliminaries}

Consider an $n$-dimensional smooth manifold $M$ and $TM$ its tangent
bundle. The Local coordinates $(x^{i})$ on $M$ induce local coordinates
$(x^{i},y^{i})$ on $TM$. We will use the following notations: $C^{\infty}(TM)$
for the real algebra of smooth functions on $TM$, $\mathfrak{X}(TM)$
for the Lie algebra of smooth vector fields on $TM$ and $\Lambda^{k}(TM)$
for the $C^{\infty}(TM)$ module of differentiable $k$-forms on $TM$.

Let $\pi:TM\to M$ be the canonical submersion and $VTM:u\in TM\to V_{u}TM=\operatorname{Ker}d_{u}\pi\subset T_{u}TM$,
the vertical distribution. A vector field $X\in\mathfrak{X}(TM)$
is vertical if $X_{u}\in V_{u}TM,\,\,\forall u\in TM$.

Consider $\mathbf{C}\in\mathfrak{X}(TM)$ the Liouville vector field
and $J$ the tangent structure (vertical endomorphism), locally given
by: 
\begin{eqnarray*}
\mathbf{C}=y^{i}\frac{\partial}{\partial y^{i}},\quad J=\frac{\partial}{\partial y^{i}}\otimes dx^{i}.
\end{eqnarray*}

A \emph{semi-spray}, or a second order vector field, is a globally
defined vector field on $TM$, $S\in\mathfrak{X}(TM)$, that satisfies
$JS=\mathbf{C}$. In local coordinates, a semi-spray $S$ can be expressed
as follows: 
\begin{eqnarray*}
S=y^{i}\frac{\partial}{\partial x^{i}}-2G^{i}(x,y)\frac{\partial}{\partial y^{i}},
\end{eqnarray*}
where $G^{i}(x,y)$ are smooth functions on $TM$.

Any system of SODE in normal form (\ref{eq:sode}) can be identified
with a semi-spray on $TM$, with local coefficients $G^{i}(x,y)$.

A \emph{spray} is a homogeneous semi-spray of degree $2$ with respect
to fiber coordinates $y$. This means that $[\mathbf{C},S]=S$ $\Leftrightarrow$
the functions $G^{i}(x,y)$ are homogeneous of degree 2: $G^{i}(x,ry)=r^{2}G^{i}(x,y)\,\,\forall r\in\mathbb{R}$.

A curve $c:I\to M$ is called \emph{regular} if its tangent lift takes
values in the slashed tangent bundle, $c':I\to TM\setminus\{0\}$.
A regular curve is called a \emph{geodesic} of the semi-spray $S$
if $S\circ c'=c''$. Locally, $c(t)=(x^{i}(t))$ is a geodesic of
$S$ if 
\begin{eqnarray}
\frac{d^{2}x^{i}}{dt^{2}}+2G^{i}\left(x,\frac{dx}{dt}\right)=0.\label{d2g}
\end{eqnarray}

A \emph{Lagrangian} on $TM$ is a smooth function $L:TM\to\mathbb{R}$
whose Hessian with respect to the fiber coordinates 
\begin{eqnarray}
g_{ij}=\frac{\partial^{2}L}{\partial y^{i}\partial y^{j}}\label{gij}
\end{eqnarray}
is non-trivial. $L$ is said to be \emph{regular} if the Poincar\'e-Cartan
$2$-form $dd_{J}L$ is a symplectic form on $TM$, where $d_{J}=i_{J}\circ d-d\circ i_{J}$.
Locally, this is equivalent to the fact that the metric (\ref{gij})
has maximal rank $n$ on $TM$.

For a Lagrangian $L$, we consider $E_{L}=\mathbf{\mathbb{\mathbf{C}}}(L)-L$
its \emph{Lagrangian energy}.\\

A \emph{semi-basic} form on $TM$ is a form on $TM$ which vanishes
whenever one of its arguments is vertical. A \emph{semi-basic} vector
valued form $\mu$ on $TM$ of degree $m$ is a map $\mu:\bigotimes^{m}\mathfrak{X}(TM)\rightarrow\mathfrak{X}(VTM)$
which vanishes whenever one of its arguments is vertical. For example,
the tangent structure $J$ is a vector valued semi-basic $1$-form.

For an arbitrary semi-spray $S$ and a Lagrangian $L$, we consider
the following semi-basic $1$-form\emph{, }called\emph{ the Lagrange
differential} \cite{key-12}: 
\begin{eqnarray*}
\delta_{S}L & = & {\mathcal{L}_{S}}d_{J}L-dL=\left\{ S\left(\frac{\partial L}{\partial y^{i}}\right)-\frac{\partial L}{\partial x^{i}}\right\} dx^{i}.
\end{eqnarray*}

We introduce next some of the derivations within the Frölicher-Nijenhuis
formalism, \cite{key-7}, which will be used in sequel.

Consider a vector valued $q$-form $Q$ on $TM$. Let $i_{Q}:\Lambda^{k}(TM)\to\Lambda^{k+q-1}(TM)$
be the derivation of degree $(q-1)$, given by 
\begin{eqnarray*}
i_{Q}\theta(X_{1},...,X_{k+p-1})=\frac{1}{q!(k-1)!}\sum_{\sigma\in S_{k+q-1}}\operatorname{sign}(\sigma)\theta\left(Q(X_{\sigma(1)},...,X_{\sigma(q)}),X_{\sigma(q+1)},...,X_{\sigma(k+q-1)}\right),
\end{eqnarray*}
where $S_{k+q-1}$ is the permutation group of $\{1,..,k+q-1\}$.
In particular, for the vertical endomorphism $J$ and $\theta\in\Lambda^{1}(TM)$,
we obtain $\left(i_{J}\theta\right)(X)=\theta(JX),$ for all $X\in\mathfrak{X}(TM)$.
Note that $i_{Q}$ is trivial on functions.

Let $d_{Q}:\Lambda^{k}(TM)\to\Lambda^{k+q}(TM)$ be the derivation
of degree $q$, which given by 
\begin{eqnarray*}
d_{Q}=i_{Q}\circ d-(-1)^{q-1}d\circ i_{Q}.
\end{eqnarray*}
For instance, $\left(d_{J}\theta\right)(X,Y)=\left(JX\right)(\theta Y)-\left(JY\right)(\theta X)-\theta[JX,Y]+\theta[JY,X]-\theta\left(J[X,Y]\right)$,
$\forall X,Y\in\mathfrak{X}(TM),\,\,\forall\theta\in\Lambda^{1}(TM)$.
Evidently $d_{X}=\mathcal{L}_{X}$ is the Lie derivative if $X\in\mathfrak{X}(TM)$.
The derivation $d_{Q}$ commutes with the exterior derivative $d$.

For two vector valued forms $K$ and $Q$ on $TM$, of degrees $k$
and $q$, respectively, we consider the Frï¿œlicher-Nijenhuis bracket
$[K,Q]$, which is the vector valued $(k+q)$-form, uniquely determined
by 
\begin{eqnarray}
d_{[K,Q]}=d_{k}\circ d_{Q}-(-1)^{kq}d_{Q}\circ d_{K}.\label{dkl}
\end{eqnarray}
For example, $[X,J]=\mathcal{L}_{X}J=\mathcal{L}_{X}\circ J-J\circ\mathcal{L}_{X}$,
$\forall X\in\mathfrak{X}(TM)$.

The definition of the tangent structure $J$ leads to $[J,J]=0$ and
that $d_{J}^{2}=0$ follows from the formula \eqref{dkl}. Therefore,
any $d_{J}$-exact form is $d_{J}$-closed and according to a Poincar\'e-type
Lemma \cite{13-1}, any $d_{J}$-closed form is locally $d_{J}$-exact.

\begin{defn} Given a semi-spray $S$, we say that $S$ is $Lagrangian$
if there exists a Lagrangian $L$ such that $\delta_{S}L=0$. This
means that the solutions of the system (\ref{eq:sode}) are among
the solutions of the Euler-Lagrange equations of some Lagrangian $L$
: $\frac{d}{dt}\left(\frac{\partial L}{\partial\dot{x}^{i}}\right)-\frac{\partial L}{\partial x^{i}}=0$.
If $L$ is regular, these two systems are equivalent \cite{key-2-1}.

\end{defn}

\section{Deformations of Lagrangians: main results}

Consider a SODE in normal form (\ref{eq:sode}), which can be written
in an equivalent form as in (\ref{eq:eqnoncons}), for some Lagrangian
$L$. This system is equivalent to $\delta_{S}L=\sigma$, where $S$
is the semi-spray on $TM$ with local coefficients $G^{i}(x,y)$ and
$\sigma$ is a semi-basic 1-form $\sigma=\sigma_{i}(x,y)dx^{i}$. 

In this section we determine necessary and sufficient conditions for
the existence of non-constant, differentiable of class $C^{2}$ deformations
$\Phi:\mathbb{R}\rightarrow\mathbb{R}$, such that the system (\ref{eq:sode})
is equivalent to the system of Euler-Lagrange equations of $\Phi(L)$:
\[
\frac{d}{dt}\left(\frac{\partial\Phi(L)}{\partial\dot{x}^{i}}\right)-\frac{\partial\Phi(L)}{\partial x^{i}}=0\Longleftrightarrow\delta_{S}\Phi(L)=0.
\]
Now $S$ is a Lagrangian second order vector field corresponding to
the Lagrangian $\Phi(L)$.

Using Frï¿œlicher-Nijenhuis formalism and the theory of derivations,
we obtain the main result of the paper. 

\begin{thm} \label{thm:main}Consider a SODE of type (\ref{eq:sode})
and $S\in\mathfrak{X}(TM)$ the associated semi-spray. Suppose that
the SODE can be written in an equivalent way as $\delta_{S}L=\sigma$,
for some Lagrangian $L$ and some semi-basic 1-form $\sigma=\sigma_{i}(x,\dot{x})dx^{i}$,
with $S(L)\neq0$ and $\mathbf{C}(L)\neq0$. Then there exists a non-constant,
differentiable of class $C^{2}$ function $\Phi:\mathbb{R}\rightarrow\mathbb{R}$
such that $S$ is a Lagrangian vector field with the corresponding
Lagrangian $\Phi(L)$ if and only if the following three conditions
are satisfied:

\begin{equation}
\sigma=\frac{S(E_{L})}{\mathbf{C}(L)}d_{J}L,\label{eq:3}
\end{equation}

\begin{equation}
\frac{\Phi''(L)}{\Phi'(L)}=-\frac{S(E_{L})}{S(L)\mathbf{C}(L)},\label{eq:2}
\end{equation}
\begin{equation}
\left(\Phi''\frac{\partial L}{\partial y^{j}}\frac{\partial L}{\partial y^{i}}+\Phi'\frac{\partial^{2}L}{\partial y^{i}\partial y^{j}}\right)_{i,j}\neq O_{n},\label{eq:0}
\end{equation}
where $O_{n}$ is the null matrix. \end{thm} 

\begin{proof} We first prove the forward direction implication. For
that we investigate the relation between the equations of motion for
$L$ and for its deformation $\Phi(L)$. To do so, we first compute
the Lagrange differential of the deformed Lagrangian $\Phi(L)$: 
\begin{eqnarray*}
\delta_{S}\Phi(L) & = & \left[S\left(\frac{\partial\Phi(L)}{\partial y^{i}}\right)-\frac{\partial\Phi(L)}{\partial x^{i}}\right]dx^{i}=\left[S\left(\Phi'(L)\frac{\partial L}{\partial y^{i}}\right)-\Phi'(L)\frac{\partial L}{\partial x^{i}}\right]dx^{i}\\
 & = & \left[\Phi''(L)S(L)\frac{\partial L}{\partial y^{i}}+\Phi'(L)S\left(\frac{\partial L}{\partial y^{i}}\right)-\Phi'(L)\frac{\partial L}{\partial x^{i}}\right]dx^{i}.
\end{eqnarray*}

Since $d_{J}L=\frac{\partial L}{\partial y^{i}}dx^{i}$, the above
relation is equivalent to 
\[
\delta_{S}\Phi(L)=\Phi''(L)S(L)d_{J}L+\Phi'(L)\delta_{S}L.
\]

Now, suppose that $S$ is a Lagrangian vector field with the corresponding
Lagrangian $\Phi(L)$.

The conditions $\delta_{S}L=\sigma$ and $\delta_{S}\Phi(L)=0$ are
simultaneously satisfied if and only if 
\begin{equation}
\Phi''(L)S(L)d_{J}L+\Phi'(L)\sigma=0.\label{eq:delta_Sphi(L)}
\end{equation}

Let $X=S$ and $K=J$ in the formula $i_{X}d_{K}=-d_{K}i_{X}+\mathcal{L}_{KX}+i_{[K,X]}$,
(\cite{key-7}, Appendix A, page 215). Then 
\begin{equation}
i_{S}d_{J}+d_{J}i_{S}=\mathcal{L}_{JS}+i_{[J,S]}.\label{eq:i_Sd_J}
\end{equation}

Applying $i_{S}$ to the equation (\ref{eq:delta_Sphi(L)}), we get
$\Phi''(L)=-\frac{S(E_{L})}{S(L)\mathbf{C}(L)}\Phi'(L)$. This formula
is obtained using (\ref{eq:i_Sd_J}), noting that $i_{S}d_{J}L=\mathcal{L}_{\mathbf{C}}(L)=\mathbf{C}(L)$
and{\small{ 
\[
i_{S}\sigma=i_{S}\left(\mathcal{L}_{S}d_{J}L-dL\right)=i_{S}\left(di_{S}d_{J}L+i_{S}dd_{J}L-dL\right)=i_{S}\left(d\mathbf{C}(L)-dL\right)=i_{S}dE_{L}=\left(\mathcal{L_{S}}-di_{S}\right)(E_{L})=S(E_{L}).
\]
}}{\small \par}

Substituting $\Phi''(L)$ in (\ref{eq:delta_Sphi(L)}) we obtain the
required expression for $\sigma$ in (\ref{eq:3}).

If we want $\Phi(L)$ to be also a Lagrangian, then we first need
to investigate its Hessian with respect to the fiber coordinates $(y^{i})$.

Let us denote $\tilde{g}_{ij}=\frac{\partial^{2}\Phi(L)}{\partial y^{i}\partial y^{j}}$.
It results $\tilde{g}_{ij}=\Phi''\frac{\partial L}{\partial y^{j}}\frac{\partial L}{\partial y^{i}}+\Phi'g_{ij}$.
Thus $\Phi(L)$ is a Lagrangian if there exist a deformation $\Phi$
such that $\left(\Phi''\frac{\partial L}{\partial y^{j}}\frac{\partial L}{\partial y^{i}}+\Phi'g_{ij}\right)$
is a non-trivial matrix. In other words, the non-constant function
$\Phi$ has to satisfy the condition 
\begin{equation}
\exists\, i,j\in\overline{1,n}\,\, g_{ij}\neq-\frac{\Phi''}{\Phi'}\frac{\partial L}{\partial y^{j}}\frac{\partial L}{\partial y^{i}}.\label{nontrivial}
\end{equation}

Following \eqref{nontrivial} and the definition of a Lagrangian,
$\Phi(L)$ is a regular Lagrangian if and only if 
\[
rank\left(\Phi''\frac{\partial L}{\partial y^{j}}\frac{\partial L}{\partial y^{i}}+\Phi'g_{ij}\right)=n\text{ on }TM.
\]

The converse is trivial. \end{proof}

\begin{rem} For $\sigma=0$, that is, studying the deformations of
conservative mechanical system, we reobtain the result of \cite{key-3}:
$\Phi(L)$ is dynamically equivalent to $L$ if and only if $\Phi(t)=at+b,\, a,b\in\mathbb{\mathbb{R}},\, a\neq0$
or $L$ is a constant of motion for $S$ that is $S(L)=0$. In the
second case, $\Phi$ can be any differentiable function.

\end{rem}

We give an instance of immediate application of Theorem \ref{thm:main}.
Suppose that we have a semi-spray whose equations of motions are written
as $\delta_{S}L=\sigma$. In order to check if $S$ is Lagrangian
with respect to the alternative Lagrangian $\Phi(L)$ , it is enough
to verify that condition \eqref{eq:3} is satisfied. In the affirmative
case, we have to integrate equation (\ref{eq:2}) and obtain the deformation
$\Phi$. But this is possible only if the expression $-\frac{S(E_{L})}{S(L)\mathbf{C}(L)}$
can be written as an integrable function of the initial Lagrangian
$L$. Also, the matrix (\ref{eq:0}) has to be non-trivial, respectively
of maximal rank $n$, depending on what kind of Lagrangian $\Phi(L)$
we search for, a singular or a regular one.

In the last section we will consider some classes of Lagrangian's
deformations, inspired by the references, that satisfy the conditions
\[
\delta_{S}L=\frac{S(E_{L})}{\mathbf{C}(L)}d_{J}L,\,\,\exists\, f:\,\mathbb{R}\rightarrow\mathbb{R}:\,-\frac{S(E_{L})}{S(L)\mathbf{C}(L)}=f(L)
\]
and condition (\ref{eq:0}), where $f\,:\mathbb{\, R\rightarrow\mathbb{R}}$
is an integrable function.

If we suppose that $\Phi$ is a strictly increasing function, then
\[
\Phi(L)=\int\exp\left(\int f(L)dL\right)dL,\,\,\delta_{S}\Phi(L)=0.
\]

\subsection*{Dissipative Forces}

Here we restrict Theorem 1 to the particular case when $\sigma$ is
given by a $d_{J}$-exact, semi-basic 1-form $\sigma$, which means
$\sigma$ = $d_{J}D$, for a function $D$ on $TM$. This case correspond
to dissipative mechanical systems. It should be noted that it includes
the classic dissipation of Rayleigh type \cite{Goldstein}, when the
function $D$ is negative definite, quadratic in the velocities.

If $\sigma$ = $d_{J}D$, the system of Lagrange equations can be
written as follows: 
\begin{equation}
\frac{d}{dt}\left(\frac{\partial L}{\partial\dot{x}^{i}}\right)-\frac{\partial L}{\partial x^{i}}=\frac{\partial D}{\partial\dot{x}^{i}}.\label{diss_EL_eq.}
\end{equation}

We want to recast a dissipative system into the form of original Euler-Lagrange
equations. One would try to find a different function $\Phi(L)$ such
that the Euler-Lagrange equations of $\Phi(L)$ are equivalent to
the given system.

One can easily check that in the dissipative case $S(E_{L})=\mathbb{\mathbf{C}}(D)$,
hence the first two conditions of Theorem 1 can be written as follows:
\[
\sigma=\frac{\mathbb{\mathbf{\mathbf{C}}}(D)}{\mathbf{\mathbf{\mathbb{\mathbf{C}}}}(L)}d_{J}L,\quad\frac{\Phi''(L)}{\Phi'(L)}=-\frac{\mathbb{\mathbf{\mathbf{C}}}(D)}{S(L)\mathbf{\mathbb{\mathbf{\mathbb{\mathbf{C}}}}}(L)}.
\]
For the Rayleigh type dissipation, using 
\[
S(E_{L})=2D<0,
\]
we can slightly adapt the two above conditions. 

The next example is from \cite{key-9}.

\begin{example} Let $M$ be a real, $2$-dimensional, differentiable
manifold. Consider the SODE 
\begin{equation}
\begin{cases}
\frac{d^{2}x^{1}}{dt^{2}}+ax^{1}+bx^{2}+\omega y^{1} & =0,\\
\frac{d^{2}x^{2}}{dt^{2}}-bx^{1}+ax^{2}-\omega y^{2} & =0,
\end{cases}\label{eq:dissipative}
\end{equation}
$a,b,\omega\in\mathbb{R}^{*}$ and its associated semi-spray $S=y^{1}\frac{\partial}{\partial x^{1}}+y^{2}\frac{\partial}{\partial x^{2}}-\left(ax^{1}+bx^{2}+\omega y^{1}\right)\frac{\partial}{\partial y^{1}}-\left(-bx^{1}+ax^{2}-\omega y^{2}\right)\frac{\partial}{\partial y^{2}}$.

We consider the kinetic energy Lagrangian 
\[
L(x,y)=\frac{1}{2}\left[(y^{1})^{2}+(y^{2})^{2}\right]
\]
and the semi-basic 1-form 
\[
\sigma=-\frac{1}{(y^{1})^{2}+(y^{2})^{2}}\left(ax^{1}y^{1}+bx^{2}y^{1}+\omega(y^{1})^{2}-bx^{1}y^{2}+ax^{2}y^{2}-\omega(y^{2})^{2}\right)\left(y^{1}dx^{1}+y^{2}dx^{2}\right).
\]
The SODE (\ref{eq:dissipative}) is equivalent with $\delta_{S}L=\sigma$.
Remark that $\sigma=d_{J}D$, with 
\[
D=-a\left(x^{1}y^{1}+x^{2}y^{2}\right)+b\left(x^{1}y^{2}-x^{2}y^{1}\right)+\frac{1}{2}\omega\left((y^{2})^{2}-(y^{1})^{2}\right).
\]
It is clear that $\mathbb{\mathbf{\mathbb{\mathbf{C}}}}(L)=2L$ which
implies $E_{L}=L$. Moreover, one can easily verify that $D$ satisfies
the condition $S(E_{L})=\mathbb{\mathbf{\mathbb{\mathbf{C}}}}(D)$
as follows: 
\begin{align*}
d_{J}L & =y^{1}\, dx^{1}+y^{2}\, dx^{2},\,\,\,\,\,\,\,\, S(L)=S(E_{L})=-y^{1}(ax^{1}+bx^{2}+wy^{1})+y^{2}(bx^{1}-ax^{2}+wy^{2})=\mathbb{\mathbf{\mathbb{\mathbf{C}}}}(D).
\end{align*}

Furthermore, we obtain the next equation $\frac{\Phi''(L)}{\Phi'(L)}=-\frac{1}{2L}$.

If we are searching for functions $\Phi$ such as $\Phi'>0$, it follows
that $\ln\Phi'=-\frac{1}{2}\ln L$ and hence $\Phi(L)=a\sqrt{L}+b$,
with $a,b\in\mathbb{R}$ some constants, $a>0$.

We can easy verify that $S$ is a Lagrangian vector field with its
associated Lagrangian 
\[
\tilde{L}(x,y)=a\sqrt{\left(y^{1}\right)^{2}+\left(y^{2}\right)^{2}}+b\quad\text{ and }\quad\delta_{S}\tilde{L}=0.
\]
\end{example}

\subsection*{Homogeneous Lagrangians}

Suppose that $S$ is a spray of Lagrangian type with covariant force
field $\sigma$ and the corresponding Lagrangian $L$. In this subsection,
we assume that $\sigma$ and $L$ are both homogeneous of order $p$.
It means that $[\mathbf{\mathbb{\mathbf{C}}},L]=\mathbf{\mathbf{\mathbb{\mathbf{C}}}}(L)=pL$
and $[\mathbf{C},\sigma]=\mathcal{L}_{\mathbf{C}}\sigma=p\sigma$.
We will work on $T_{0}M$, the tangent bundle with the zero section
removed.

If $p=1\Leftrightarrow\mathbf{\mathbf{\mathbb{\mathbf{C}}}}(L)=L$,
then $E_{L}=0$ and formula (\ref{eq:3}) implies $\sigma=0$. This
case was treated in \cite{key-3}. Therefore, we assume $p>1$. 

\begin{thm} \label{thm:hom}Let $S\in\mathfrak{X}(T_{0}M)$ be a
spray. Suppose that there exists a Lagrangian $L$ with positive values
and a semi-basic 1-form $\sigma$ on $T_{0}M$, such that $\delta_{S}L=\sigma$,
with both $L$ and $\sigma$ homogeneous of order $p>1$, Then there
exists a differentiable of class $C^{2}$, strictly increasing function
$\Phi:\mathbb{R}\rightarrow\mathbb{R}$, such that $S$ is a Lagrangian
vector field with the corresponding Lagrangian $\Phi(L)$, if and
only if 
\[
d_{J}L\wedge\sigma=0
\]
 and 
\begin{equation}
\left(\frac{1-p}{p}\frac{\partial L}{\partial y^{j}}\frac{\partial L}{\partial y^{i}}+L\frac{\partial^{2}L}{\partial y^{i}\partial y^{j}}\right)\neq O_{n}.\label{gijhomogenous}
\end{equation}
Moreover, the deformation $\Phi$ is given by 
\[
\Phi(L)=aL^{\frac{1}{p}}+b,\,\, a,b\in\mathbb{R}.
\]
\end{thm} \begin{proof} Suppose that $S$ is a Lagrangian vector
field with the corresponding Lagrangian $\Phi(L)$: $\delta_{S}\Phi(L)=0$.
From the above theorem, using $\mathbf{\mathbb{C}}(L)=pL$ and $E_{L}=(p-1)L$,
we deduce 
\begin{equation}
\frac{\Phi''}{\Phi'}=\left(\frac{1}{p}-1\right)\frac{1}{L}\label{eq:5}
\end{equation}

and 
\begin{equation}
\sigma=\left(1-\frac{1}{p}\right)S\left(\ln L\right)d_{J}L.\label{eq:sigma-hom}
\end{equation}
It follows that $d_{J}L\wedge\sigma=0$.

Conversely, suppose that $\sigma$ is homogeneous of order $p>1$
and $d_{J}L\wedge\sigma=0$. Applying $i_{S}$ to this relation, it
yields 
\[
\mathbf{\mathbb{\mathbf{C}}}(L)\sigma-i_{S}\sigma d_{J}L=0\Longleftrightarrow pL\sigma-(p-1)S(L)d_{J}L=0\Longleftrightarrow\sigma=\left(1-\frac{1}{p}\right)\frac{S\left(L\right)}{L}d_{J}L.
\]
Hence $\sigma=\frac{S(E_{L})}{\mathbf{C}(L)}d_{J}L$. Using again
the last theorem, it results that $\delta_{S}\Phi(L)=0$.

For $L$ with positive values and $\Phi$ a strictly increasing function,
we can easily integrate the equation (\ref{eq:5}) and obtain 
\[
\Phi(L)=aL^{\frac{1}{p}}+b,\,\, a,b\in\mathbb{R}.
\]

The fact that the Hessian matrix of $\Phi(L)$ is non-trivial is equivalent
to the condition (\ref{gijhomogenous}).

\end{proof} 

Note that we started with a Lagrangian $L$ homogeneous of order $p>1$,
and for $b=0$, the deformed Lagrangian $\Phi(L)$ is homogeneous
of order 1, $\mathbf{C}\left(\Phi(L)\right)=\Phi(L)$.

\section{Examples }

All the examples of this section satisfy the conditions 
\begin{equation}
\delta_{S}L=\frac{S(E_{L})}{\mathbf{C}(L)}\, d_{J}L,\,\,\exists\, f:\,\mathbb{R}\longrightarrow\mathbb{R}:\,-\frac{S(E_{L})}{S(L)\mathbf{C}(L)}=f(L),\label{eq:*-1}
\end{equation}
with $f$ an integrable function. The examples are divided into the
following cases:\\

\begin{enumerate}
\item \cite{C-N,key-7-1} If $f$ is a constant function, $f(L)=\gamma=\text{constant}$,
it follows $\Phi(L)=\frac{1}{\gamma}\exp(\gamma L)+a,\, a\in\mathbb{R}$. 
\item \cite{key-1,C-N,key-7-1,key-8-1} If $f(L)=\frac{\gamma}{L+a},\,\,\gamma\in\mathbb{R}^{*}\backslash\{-1\}$,
$a\in\mathbb{R}$, then $\Phi(L)=\frac{1}{1+\gamma}(L+a)^{1+\gamma}+b$,
$b\in\mathbb{R}$. 

\begin{enumerate}
\item As a particular subcase, $f(L)=\left(\frac{1}{p}-1\right)\frac{1}{L},\,\, p\in\mathbb{N},\, p>1$.
It follows $\Phi(L)=aL^{\frac{1}{p}}+b,\, a,b\in\mathbb{R}$. This
deformation is obtained when we search for Lagrangians that are homogeneous
of order $p>1$, as the Lagrangians studied in Finsler geometry \cite{key-0,key-1}.
In fact, the most important class of homogeneous Lagrangians is the
one of Finsler functions which have $p=2$. 
\item {\small{The deformation $\Phi(L)=\frac{b}{2}L^{2}+aL+c$ was used
by }}Kawaguchi, T. and Miron, R. \cite{key-8-1}. They studied the
geometry of a generalized Lagrange space ( obtained from a Riemannian
space) using a deformation of the Riemannian metric. 
\end{enumerate}
\item \cite{key-1} If $f(L)=-\frac{1}{L+a}$, $a\in\mathbb{R}$, then it
follows $\Phi(L)=b\ln(L+a)+c$, $b,c\in\mathbb{R}$. 
\item \cite{key-1} If $f(L)=-\frac{2c}{cL+d},\: c,d\in\mathbb{R}$, then
it follows $\Phi(L)=\frac{aL+b}{cL+d},\, a,b\in\mathbb{R},\,\, ad-bc\neq0$. 
\end{enumerate}
We start with an example corresponding to the first class of Lagrangian's
deformations enumerated above.

\begin{example} Let $M$ be a real, $3$-dimensional, differentiable
manifold. Consider the SODE 
\begin{equation}
\begin{cases}
\frac{d^{2}x^{1}}{dt^{2}} & =0,\\
\frac{d^{2}x^{2}}{dt^{2}} & =0,\\
\frac{d^{2}x^{3}}{dt^{2}}+x^{1}y^{1} & =0.
\end{cases}\label{case1}
\end{equation}

The associated semi-spray is $S=y^{1}\frac{\partial}{\partial x^{1}}+y^{2}\frac{\partial}{\partial x^{2}}+y^{3}\frac{\partial}{\partial x^{3}}-x^{1}y^{1}\frac{\partial}{\partial y^{3}}$.
We consider the regular Lagrangian 
\[
L(x,y):=a+\frac{1}{b}\ln\left[b\left(x^{1}y^{1}+x^{2}y^{2}+x^{3}y^{3}+(y^{1})^{2}+(y^{2})^{2}\right)-c\right],
\]

with $a,b,c\in\mathbb{R}^{*}$, defined on an open set such that $x^{1}y^{1}+x^{2}y^{2}+x^{3}y^{3}+(y^{1})^{2}+(y^{2})^{2}>\frac{c}{b}$
and the semi-basic 1-form 
\[
\,\sigma=\frac{b\left[x^{1}x^{3}y^{1}-(y^{1})^{2}-(y^{2})^{2}-(y^{3})^{2}\right]}{[b(x^{1}y^{1}+x^{2}y^{2}+x^{3}y^{3}+(y^{1})^{2}+(y^{2})^{2})-c]^{2}}\left[(x^{1}+2y^{1})dx^{1}+(x^{2}+2y^{2})dx^{2}+x^{3}dx^{3}\right].
\]

By direct computations, we obtain that the system (\ref{case1}) is
equivalent to $\delta_{S}L=\sigma$ with $L$ and $\sigma$ defined
above. We use 
\begin{align*}
S(L) & =\exp\left(b(a-L)\right)\left[(y^{1})^{2}+(y^{2})^{2}+(y^{3})^{2}-x^{1}x^{3}y^{1}\right],\\
d_{J}L & =\exp\left(b(a-L)\right)\left[(x^{1}+2y^{1})dx^{1}+(x^{2}+2y^{2})dx^{2}+x^{3}dx^{3}\right],\\
\mathbf{C}(L) & =\exp\left(b(a-L)\right)\left[x^{1}y^{1}+x^{2}y^{2}+x^{3}y^{3}+2(y^{1})^{2}+2(y^{2})^{2}\right],\\
S(E_{L}) & =-b\mathbb{C}(L)\, S(L),\,\,\,\,\,\,\,\,\,\delta_{S}L=-bS(L)\, d_{J}L.
\end{align*}
We verify that both conditions (\ref{eq:*-1}) are satisfied, with
$f(L)=b$. Hence, the deformed Lagrangian is given by 
\[
\Phi(L)(x,y)=\frac{1}{b}\exp\left(ab\right)\left[b\left(x^{1}y^{1}+x^{2}y^{2}+x^{3}y^{3}+(y^{1})^{2}+(y^{2})^{2}\right)-c\right].
\]
One can easily see that $\Phi(L)$ is a regular Lagrangian. \end{example} 

The next example corresponds to the second class of Lagrangian's deformations.

\begin{example} \cite{key-5,key-8} Consider the Li\'enard-type
second-order non-linear differential equation of the form 
\begin{equation}
\frac{d^{2}x}{dt^{2}}+g(x)\frac{dx}{dt}+h(x)=0,\label{eq:Lienard}
\end{equation}
where $g$ and $h$ are real, smooth functions of $x$, defined on
an interval. It plays an important role in many areas of applied sciences,
cardiology, neurology, biology, mechanics, seismology, chemistry,
physics, and cosmology.

If the functions $g,\, h$ satisfy the Chiellini integrability condition
\[
\frac{d}{dx}\left(\frac{h}{g}\right)=kg,
\]
with $k\in\mathbb{R}^{*}$, then it is possible to construct exact
solutions for a first-order Abel equation of the first kind associated
to (\ref{eq:Lienard}).

We consider the Li\'enard equation (\ref{eq:Lienard}) with $g,\, h$
satisfying Chiellini's condition with $k$ of the form $-\alpha(\alpha+1)$
and $\alpha\in\mathbb{R}^{*}$.

Hence, equation (\ref{eq:Lienard}) can be written in the form $\delta_{S}L=\sigma$,
with $S=y\frac{\partial}{\partial x}-\left(g(x)y+h(x)\right)\frac{\partial}{\partial y}$
a semi-spray on an $1$-dimensional, real, differentiable manifold,
$\sigma=2\left(\frac{1}{\alpha}h(x)-g(x)y\right)dx$ a semi-basic
1-form and the Lagrangian 
\[
L(x,y)=\left(y-\frac{1}{\alpha}\frac{h(x)}{g(x)}\right)^{2}.
\]

We can show that $\sigma$ satisfies the conditions (\ref{eq:*-1}),
by the following computations: 
\begin{eqnarray*}
\mathbf{\mathbb{\mathbf{C}}}(L) & = & 2y\left(y-\frac{1}{\alpha}\frac{h(x)}{g(x)}\right),\,\,\,\,\,\,\,\,\,\,\,\,\ E_{L}=y^{2}-\frac{1}{\alpha^{2}}\left(\frac{h(x)}{g(x)}\right)^{2}\\
S(L) & = & 2\alpha g(x)L,\,\,\,\,\,\,\,\,\,\,\,\,\,\,\,\,\,\,\,\,\,\,\,\,\,\,\, S(E_{L})=-g(x)\mathbf{C}(L),\,\,\,\,\,\,\,\,\,\,\,\,\,\,\,\,\,\,\,\,\, d_{J}L=-\frac{1}{g(x)}\sigma.
\end{eqnarray*}
Moreover, $f(L)=-\frac{1}{2\alpha L}$, hence the deformed Lagrangian
is given by 
\[
\Phi(L)=\frac{2\alpha\, a}{2\alpha+1}L^{\frac{2\alpha+1}{2\alpha}}+b,\,\, a,b\in\mathbb{R}.
\]
We know from theorem 1 that for any $a,b\in\mathbb{R},$ $\Phi(L)$
verifies $\delta_{S}\Phi(L)=0$. For $b=0$ and $a=\frac{\alpha}{2(\alpha+1)}$
we get

\[
\Phi(L)(x,y)=\frac{\alpha^{2}}{(\alpha+1)(2\alpha+1)}\left(y-\frac{1}{\alpha}\frac{h(x)}{g(x)}\right)^{\frac{2\alpha+1}{\alpha}}.
\]
This Lagrangian was obtained in \cite{key-5} using the Jacobi Last
Multiplier method. \end{example}

The next example is subordinated to the third class of Lagrangian's
deformations. 

\begin{example} Let $M$ be a real, $3$-dimensional, differentiable
manifold. Consider the SODE 
\begin{equation}
\begin{cases}
\frac{d^{2}x^{1}}{dt^{2}} & =0,\\
\frac{d^{2}x^{2}}{dt^{2}}+2x^{2} & =0,\\
\frac{d^{2}x^{3}}{dt^{2}} & =0,
\end{cases}\label{case31}
\end{equation}
and the associated semi-spray is $S=y^{1}\frac{\partial}{\partial x^{1}}+y^{2}\frac{\partial}{\partial x^{2}}+y^{3}\frac{\partial}{\partial x^{3}}-2x^{2}\frac{\partial}{\partial y^{2}}$.
The SODE (\ref{case31}) is equivalent to $\delta_{S}L=\sigma$, for
the regular Lagrangian 
\[
L(x^{2},y^{1},y^{2},y^{3}):=f(y^{1})g(y^{3})\,\exp\left(a\left[\frac{1}{2}(y^{2})^{2}-(x^{2})^{2}\right]+by^{2}\right),
\]
where $f,\, g$ are smooth real functions depending only on $y^{1},\, y^{3}$,
respectively, with positive values, $a,b\in\mathbb{R}^{*}$ and the
semi-basic 1-form on $TM$ defined by 
\[
\sigma=-2x^{2}(2ay^{2}+b)\exp\left(a\left[\frac{1}{2}(y^{2})^{2}-(x^{2})^{2}\right]+by^{2}\right)\left(f'gdx^{1}+(b+ay^{2})fgdx^{2}+g'fdx^{3}\right).
\]
Using some computations we check that $\sigma$ satisfies the conditions
(\ref{eq:*-1}): 
\begin{align*}
d_{J}L & =L\left(\frac{f'}{f}(y^{1})dx^{1}+(b+ay^{2})dx^{2}+\frac{g'}{g}(y^{3})dx^{3}\right),\,\,\,\,\,\,\,\,\,\,\,\,\,\,\,\,\,\,\,\,\,\, S(L)=-2x^{2}(2ay^{2}+b)L,\\
\mathbf{\mathbb{\mathbf{C}}}(L) & =\left[y^{1}\frac{f'}{f}(y^{1})+y^{2}(ay^{2}+b)+y^{3}\frac{g'}{g}(y^{3})\right]L,\,\,\,\,\,\,\,\,\,\, S(E_{L})=\frac{\mathbf{\mathbf{C}}(L)S(L)}{L},\,\,\,\,\,\,\,\,\,\,\,\delta_{S}L=\frac{S(E_{L})}{\mathcal{\mathbf{\mathbb{\mathbf{C}}}}(L)}d_{J}L.
\end{align*}
Moreover, $f(L)=-\frac{1}{L}$. Then the deformed Lagrangian is given
by 
\[
\Phi(L)(x,y)=c\ln\left(f(y^{1})g(y^{3})\right)+a\left(a\left[\frac{1}{2}(y^{2})^{2}-(x^{2})^{2}\right]+by^{2}\right)+d,\,\, c,d\in\mathbb{R},\, c\neq0.
\]
If both functions $f''f-f^{2},\, g''g-g^{2}$ are non identically
zero, then $\Phi(L)$ is a regular Lagrangian. \end{example}

The following example correspond to the forth class of Lagrangian's
deformations. 

\begin{example} Let $M$ be a real, $3$-dimensional, differentiable
manifold. Consider the SODE 
\begin{equation}
\begin{cases}
\frac{d^{2}x^{1}}{dt^{2}} & =0,\\
\frac{d^{2}x^{2}}{dt^{2}} & =0,\\
\frac{d^{2}x^{3}}{dt^{2}} & =0.
\end{cases}\label{case42}
\end{equation}
The associated flat spray is $S=y^{1}\frac{\partial}{\partial x^{1}}+y^{2}\frac{\partial}{\partial x^{2}}+y^{3}\frac{\partial}{\partial x^{3}}$.
The SODE (\ref{case42}) is equivalent to $\delta_{S}L=\sigma$, with
the Lagrangian

\[
L(x,y):=-\frac{d}{c}-\frac{1}{ac^{2}\left[x^{1}y^{1}+y^{2}x^{2}+x^{3}y^{3}+(y^{1}y^{2}y^{3})^{2}+b\right]},
\]

and the semi-basic 1-form on TM {\small{{{ 
\[
\sigma=-\frac{2\,\left[(y^{1})^{2}+(y^{2})^{2}+(y^{3})^{2}\right]}{ac^{2}\left[x^{1}y^{1}+y^{2}x^{2}+x^{3}y^{3}+(y^{1}y^{2}y^{3})^{2}+b\right]^{3}}\left[\left(x^{1}+2y^{1}(y^{2}y^{3})^{2}\right)dx^{1}+\left(x^{2}+2y^{2}(y^{1}y^{3})^{2}\right)dx^{2}+\left(x^{3}+2y^{3}(y^{1}y^{2})^{2}\right)dx^{3}\right],
\]
}}}} where $c,\, a\in\mathbb{R}^{*}$ and $d,\, b\in\mathbb{R}$.
We choose the domain of $L$ such that $x^{1}y^{1}+y^{2}x^{2}+x^{3}y^{3}+(y^{1}y^{2}y^{3})^{2}+b\neq0$. 

By direct computations, we get 
\begin{align*}
d_{J}L & =a\left(cL+d\right)^{2}\left[\left(x^{1}+2y^{1}(y^{2}y^{3})^{2}\right)dx^{1}+\left(x^{2}+2y^{2}(y^{1}y^{3})^{2}\right)dx^{2}+\left(x^{3}+2y^{3}(y^{1}y^{2})^{2}\right)dx^{3}\right],\\
\mathbf{\mathbb{\mathbf{C}}}(L) & =a\left(cL+d\right)^{2}\,\left[5\left(y^{1}y^{2}y^{3}\right)^{2}-b-\frac{1}{a\left(cL+d\right)}\right],\,\,\,\, S(L)=a\left(cL+d\right)^{2}\left[(y^{1})^{2}+(y^{2})^{2}+(y^{3})^{2}\right],\\
S(E_{L}) & =\frac{2c}{(cL+d)}\mathbb{C}(L)S(L),\,\,\,\,\,\,\,\,\delta_{S}L=\frac{2c}{(cL+d)}S(L)d_{J}L.
\end{align*}
Hence $\sigma$ satisfies the conditions (\ref{eq:*-1}) and $f(L)=-\frac{2c}{cL+d}$.
Therefore, the deformed Lagrangian is {\small{{{ 
\[
\Phi(L)(x,y)=-ca\left[x^{1}y^{1}+x^{2}y^{2}+x^{3}y^{3}+\left(y^{1}y^{2}y^{3}\right)^{2}+b\right]+\frac{1}{c}\left[1+cda\left(x^{1}y^{1}+x^{2}y^{2}+x^{3}y^{3}+\left(y^{1}y^{2}y^{3}\right)^{2}+b\right)\right].
\]
}}}} In this case $L$ and $\Phi(L)$ are both regular Lagrangians.
\end{example} 

The last example is for the case of homogeneous Lagrangians. Actually,
it represents a special case of the conformal transformation of the
ecological metric \cite{key-0,key-1} which is given by $\overline{L}(x,y)=\exp(p\,\phi(x))[(y^{1})^{p}+....+(y^{n})^{p}]$,
where $n=dim(M)$, by taking $n=3,\, p=2,\,\phi(x)=x^{1}$.

\begin{example} Let $M$ be a real, $3$-dimensional, differentiable
manifold. Consider the SODE 
\begin{equation}
\begin{cases}
\frac{d^{2}x^{1}}{dt^{2}}-\left[(y^{1})^{2}+(y^{2})^{2}+(y^{3})^{2}\right] & =0,\\
\frac{d^{2}x^{2}}{dt^{2}} & =0,\\
\frac{d^{2}x^{3}}{dt^{2}} & =0,
\end{cases}\label{homogenousexample}
\end{equation}
and its associated spray $S=y^{1}\frac{\partial}{\partial x^{1}}+y^{2}\frac{\partial}{\partial x^{2}}+y^{3}\frac{\partial}{\partial x^{3}}+\left[(y^{1})^{2}+(y^{2})^{2}+(y^{3})^{2}\right]\frac{\partial}{\partial y^{1}}$.

We consider the Lagrangian 
\[
L(x,y)=\frac{1}{2}\exp(2x^{1})\left[(y^{1})^{2}+(y^{2})^{2}+(y^{3})^{2}\right].
\]
It is clear that (\ref{homogenousexample}) is equivalent to $\delta_{S}L=\sigma$,with
$\sigma=2y^{1}\exp(2x^{1})\left(y^{1}dx^{1}+y^{2}dx^{2}+y^{3}dx^{3}\right)=2y^{1}d_{J}L$.
Both $L$ and $\sigma$ are homogeneous of order $2$ and $\sigma\wedge d_{J}L=0$.
Applying Theorem \ref{thm:hom}, we obtain the deformed Lagrangian
\[
\Phi(L)=a\sqrt{L}+b=\frac{a}{\sqrt{2}}\exp(x^{1})\sqrt{(y^{1})^{2}+(y^{2})^{2}+(y^{3})^{2}}+b,\,\, a,b\in\mathbb{R}.
\]
One can easily show that $\Phi(L)$ is a singular Lagrangian. \end{example} 

Let us conclude our work. The differential equations associated to
a mechanical system usually contains a part related to a Lagrangian
function and a part related to a non-conservative force. We investigate
whether the given system can be seen to be equivalent with a purely
Lagrangian system, where the new Lagrangian function is found as a
deformation of the original Lagrangian. The conditions for this to
happen are further exploited for some specific classes of mechanical
systems, namely homogeneous and dissipative ones. Moreover, we give
many examples and show how the algorithms described in the proofs
of Theorem 1 and Theorem 2 work, how our method is easy to handle
and how one can get the corresponding alternative Lagrangian in the
affirmative case.

\subsection*{Acknowledgments}

We express our thanks to Professors J. F. Cariñena and I. Bucataru
for their advises and suggestions. Ebtsam H. Taha has been partially
supported by ``Erasmus+ Mobility Project for Higher Education Students
and Staff with Partner Countries'' and thanks Dr. Salah Elgendi for
fruitful discussions.

\end{document}